\documentclass[11pt]{article}
\usepackage[margin=1in]{geometry}
\usepackage{amsthm,amsmath}
\usepackage{latexsym,amssymb}
\usepackage{graphicx,float,color,fancybox,shapepar,setspace,hyperref}
\usepackage{subfigure}
\usepackage{pgf,tikz}
\usepackage[affil-it]{authblk}
\usepackage{indentfirst}
\usepackage[section]{placeins}
\hypersetup
{
colorlinks=true,
linkcolor=blue,
filecolor=blue,
urlcolor=blue,
citecolor=cyan,
}
\usetikzlibrary{arrows}
\makeatletter
\renewenvironment{proof}[1][\proofname]{\par
  \normalfont \topsep6\p@\@plus6\p@\relax
  \trivlist
  \item[\hskip\labelsep
        \itshape
    #1\@addpunct{.}]\ignorespaces
}{%
  \hfill $\square$
  \endtrivlist\@endpefalse
}
\makeatother

\newtheorem{thm}{Theorem}

\newtheorem{claim}{Claim}[section]
\newtheorem{cor}{Corollary}
\newtheorem{lem}{Lemma}

 \def\qed{\hfill\square}

\def\~{\sim}

\def\qed{ \hfill $\blacksquare$}

\begin{document}

\title{Ramsey numbers of sparse graphs versus disjoint books}
\author{Ting HUANG$^1$, Yanbo ZHANG$^2$, Yaojun CHEN$^{1,}$\footnote{Corresponding author. Email: yaojunc@nju.edu.cn}\\
{\small $^1$School of Mathematics, Nanjing University, Nanjing 210093, China}\\
{\small $^2$School of Mathematical Sciences, Hebei Normal University, Shijiazhuang 050024, China}}
 \date{}
\maketitle

\begin{abstract}

Let $B_k$ denote a book on $k+2$ vertices and $tB_k$ be $t$ vertex-disjoint $B_k$'s. For any integers $k\ge2$ and $t\ge1$, there exists a positive constant $\epsilon=\epsilon(k,t)$ such that every connected graph $G$ with $n\ge111t^3k^3$ vertices and at most $n(1+\epsilon)$ edges satisfies $$r(G,tB_k)=2n+t-2.$$ Our result extends the work of Erd\H{o}s, Faudree, Rousseau, and Schelp (1988), who established the corresponding result for $G$ being a tree and $t=1$.

\end{abstract}

\section{Introduction}

 Given two graphs $G$ and $H$, the \emph{Ramsey number} $r(G, H)$ is defined as the smallest positive integer $N$ such that, for every red-blue edge coloring of the complete graph $K_N$, there exists either a red subgraph isomorphic to $G$ or a blue subgraph isomorphic to $H$. When both $G$ and $H$ are complete graphs, $r(G, H)$ is the \emph{classical Ramsey number}. The \emph{chromatic number} of a graph $G$ is denoted by $\chi(G)$, and the \emph{chromatic surplus} of $G$ is defined as the cardinality of the smallest color class, taken over all proper colorings of $G$ using $\chi(G)$ colors. Burr~\cite{Burr1981} observed the following result. 

\begin{thm}[Burr~\cite{Burr1981}] \label{lowerbound}
    Let $G$ be a connected graph on $n$ vertices, and $H$ be a graph with chromatic number $\chi$ and chromatic surplus $s$. If $n\ge s$, then
    \begin{equation}\label{lower}
        r(G, H)\ge (n-1)(\chi-1)+s\,.
    \end{equation}
\end{thm}
To establish a unified formula for the Ramsey numbers of a broader class of graphs, Burr~\cite{Burr1981} introduced the notion of Ramsey goodness. A connected graph $G$ is said to be \emph{$H$-good} if equality holds in inequality~(\ref{lower}). 
One of the most classical results in this area is due to Chv\'atal~\cite{Chvatal1977}, who proved that all trees $T_n$ are $K_k$-good.
\begin{thm}[Chv\'atal~\cite{Chvatal1977}]\label{Chv}
	$r(T_n,K_k)=(n-1)(k-1)+1\,.$
\end{thm}
A tree is a minimally connected graph. Burr, Erd\H{o}s, Faudree, Rousseau, and Schelp~\cite{Burr1980} further showed that the Ramsey number remains the same if $T_n$
 is replaced by a sparse connected graph $G$ on $n$ vertices. 
\begin{thm}[Burr, Erd\H{o}s, Faudree, Rousseau, and Schelp~\cite{Burr1980}] \label{GK_m}
For $k\ge 3$, there is a positive constant $\varepsilon$ such that if $G$ is a connected graph with $n$ vertices and at most $n+\varepsilon n^{\frac{2}{k-1}}$ edges, then for $n$ sufficiently large, 
	$r(G,K_k)=(n-1)(k-1)+1$.
\end{thm}
Luo and Peng extended Theorem \ref{Chv} in another direction by replacing one $K_k$ with $t$ copies of $K_k$, and proved that all large trees are also $tK_k$-good.
\begin{thm}[Luo and Peng~\cite{Luo2023}]
	For any $t\ge 1$ and $k\ge 2$, if $n$ is sufficiently large, then $r(T_n,tK_k)=(n-1)(k-1)+t$.
\end{thm}

Recently, Zhang and Chen \cite{ZC} tried to generalize Theorem \ref{Chv} in both directions by
investigating which sparse graphs are $tK_k$-good and established the following.
\begin{thm} (Zhang and Chen \cite{ZC}) 
	Let $G$ be a connected graph with $n$ vertices and $n+\ell-2$ edges. For any integers $k\ge 2$ and $t\ge 1$, there exists a positive constant $c$ such that if $1\le \ell \le cn^{\frac{2}{k-1}}$ and $n$ is sufficiently large, then $r(G,tK_k)=(n-1)(k-1)+t$.
\end{thm}

For two graphs $G$ and $H$, $G\cup H$ denotes their disjoint union, and $G + H$ represents the disjoint union of two
 graphs $G$ and $H$, together with the new edges connecting every vertex of $G$  to every vertex of $H$.
Let $B_k^{(p)}$ denote the graph consisting of $k$ copies of $K_{p+1}$ all sharing a common $K_p$, that is, $B_k^{(p)}=K_p+\overline{K}_k$, which is called
the book graph. Here $\overline{K}_k$ is the complement of a complete graph $K_k$ on $k$ vertices.
In the study of classical Ramsey numbers, $B_k^{(p)}$ has played a pivotal role. Recently, $B_k^{(p)}$ has been found an unexpected application in the groundbreaking work of Campos, Griffiths, Morris, and Sahasrabudhe \cite{CGMS2023}, who proved for some $\epsilon > 0$, 
$r(K_n,K_n)\leq (4-\epsilon)^n$
 if $n$ is large enough, for which a key step of the proof is to find large books, as detailed in the Book Algorithm (Section 3 of \cite{CGMS2023}).

Let $B_k$ denote $B_k^{(2)}$,  
i.e., $B_k=K_2+\overline{K}_k$. Also, $B_k=K_1+K_{1,k}$.
The study of Ramsey numbers of $B_k$ versus other graphs also has long attracted significant attention, with new results continually emerging. As early as 1978, Rousseau and Sheehan \cite{RS} showed that the star on $n$ vertices is $B_k$-good. 
\begin{thm} (Rousseau and Sheehan \cite{RS})\label{bookstar}
    For any $k\geq 1$, if $n\geq 3k-3$, then
    $r(K_{1,n-1},B_k)=2n-1$.
\end{thm}
In 1988, Erd\H{o}s, Faudree, Rousseau, and Schelp \cite{EFRS0} showed that any large trees are also $B_k$-good. 
\begin{thm}(Erd\H{o}s, Faudree, Rousseau, and Schelp \cite{EFRS0})
  For any $k \geq 1$, if $n\geq 3k-3$, then
  $r(T_n,B_k)=2n-1$.
\end{thm}

Following the study of tree-book Ramsey numbers, significant attention has shifted to cycle-book Ramsey numbers. There are several results on $r(C_n, B_k)$. In 1991, Faudree, Rousseau, and Sheehan \cite{FRS 1991} proved that $C_n$ is $B_k$-good for $n\geq 2k+2$. The condition $n\geq 2k+2$ was improved by Shi \cite{Shi 2010} to $n> (6k+7)/4$. In 2021, Lin and Peng \cite{LinPeng 2021} studied $r(C_n, B_k)$ when $n$ and $k$ are nearly equal. They showed that if $n\geq 1000$, then $C_n$ is $B_k$-good for $n= \lceil 3k/2 \rceil+1$ and $C_n$ is not $B_k$-good for $8k/9+112\leq n\leq  (3k+1)/2$. In 2025, Hu, Lin, \L{}uczak, Ning, and Peng \cite{HLLNP 2025} determined
 the value of $r(C_n, B_k)$, provided that $n$ is linear with $k$ and $n$ is large enough, yielding improved bounds for cycle-book Ramsey numbers.

It is obvious that trees and cycles are both connected sparse graphs. Motivated by Theorem \ref{GK_m}, a natural and more general question is: whether a connected sparse graph $G$ is also $B_k$-good? In this paper, we answer this question in affirmative and obtain the following.
\begin{thm} \label{mr}
 Let $k$ be a positive integer and $n\geq 34k^3$. If $G$ is a connected graph with $n$ vertices
and at most $ n(1+1/(119 k^2 + 62k))$ edges, then
$r(G,B_k)=2n-1$.
\end{thm}
Furthermore, if $B_k$ is replaced by the disjoint union of copies of a $B_k$, an intriguing question arises: are trees $T_n$ still $tB_k$-good? This problem has only recently begun to attract attention. In 2024, Guo, Hu and Peng \cite{GHP} demonstrated that any tree $T_n$ is $2B_2$-good for $n\ge 5$.
We establish a general result addressing this problem as follows. First, we prove that the star $K_{1,n-1}$ is $tB_k$-good.

\begin{thm}\label{t bookstar}
Let $k,t$ be positive integers and $n\geq 3tk+3t-5$. Then
    $$r(K_{1,n-1},tB_k)= 2n+t-2.$$
\end{thm}
Second, we show that all large sparse graphs are still \emph{$tB_k$-good}.
\begin{thm} \label{mr2}
  Let $k\geq2$ and $t\geq1$ be integers, and let $n\geq 111t^3k^3$. If $G$ is a connected graph with $n$ vertices
and at most $ n(1+1/(127 t^2 k^2 + 79 t^2 k))$ edges, then
$$r(G,tB_k)=2n+t-2.$$
\end{thm}

We organize the remaining content as follows. In Section~\ref{section2}, we introduce some basic lemmas needed for the proofs. In Sections~\ref{section3}, \ref{section4}, and \ref{section5}, we present the proofs of Theorems \ref{mr}, \ref{t bookstar}, and \ref{mr2}, respectively.

We conclude this section by introducing some additional notation. Throughout this paper, all graphs are finite and simple without loops.   
For a graph $G=(V(G),E(G))$, we use $|G|$ and $e(G)$ to denote the number of vertices and edges, respectively. 
For $v\in V(G)$, the neighborhood of $v$ is $N_G(v)=\{u~|~u~\text{is adjacent to}~v\}$ and $d(v)=|N_G(v)|$. The minimum degree of $G$ is $\delta(G)=\min\{d(v)~|~v\in V(G)\}$.
For two graphs $G$ and $H$, $G-H$ denotes the subgraph of $G$ induced by the vertices of $G$ not in $H$.
Given a red-blue edge-coloring of $K_N$, $K_N[R]$ and $K_N[B]$ denote the spanning subgraphs formed by the red edges and the blue edges, respectively.
 The red and blue neighborhoods of a vertex $v$, denoted by $N_R(v)$ and $N_B(v)$ respectively, are defined as the sets of vertices adjacent to $v$ via red edges and blue edges. For a complete graph $K_N$ and a vertex subset $A$, we use $K_N[A]$ to denote
 the subgraph induced by $A$. A path $P$ of $G$ is a suspended path if each vertex of $P$, except for its endvertices, has degree $2$ in $G$. An end-edge is one incident with a vertex of degree 1.

\section{Some Basic Lemmas} \label{section2}
Our proofs build on a trichotomy for sparse graphs. We first use the following leaf-count estimate of Burr, Erd\H{o}s, Faudree, Rousseau, and Schelp, with components that are cycles excluded:

\begin{lem} (Burr, Erd\H{o}s,  Faudree, Rousseau, and Schelp \cite{BEFRS})  \label{old}
 Let $G$ be a graph on $n$ vertices and $n+\ell$ edges, with no component that is a cycle. If $G$ has no isolated
vertices and no suspended path with more than $q$ vertices, then $G$ has at least
$\lceil \frac{n}{2q}-\frac{3\ell}{2}\rceil$ vertices of degree 1.   
\end{lem} 

Zhang and Chen \cite{ZC} recently developed an enhanced version of the Trichotomy Lemma for sparse graphs, which provided the clearer structure of a sparse graph.
\begin{lem}(Zhang and Chen \cite{ZC}) \label{trichotomy}
Let $G$ be a connected graph with $n$ vertices and $n+\ell$ edges that is not a star, where $\ell \geq -1$, $n\geq q \geq 3$, and $s\geq2$ is an integer. If $G$ contains neither a suspended path of order $q$ nor a matching consisting of $s$ end-edges, then the number of vertices of degree at least 2 in $G$ is at most $\gamma$, and 
$G$ has a vertex adjacent to at least $\lceil \frac{n-\gamma}{s-1} \rceil$ vertices of degree 1, where $\gamma=(q-2)(2s+3\ell-2)+1$.
\end{lem}
 By using the Trichotomy Lemma for sparse graphs as above, we divide the problem into three cases. In each case, we use Lemma \ref{weak} to identify a subgraph $H$ of $G$, where $G$ is the given sparse graph. Two tools are utilized to extend $H$ into the desired graph $G$. The first case relies on the following path extension lemma, in the form given by Erd\H{o}s, Faudree, Rousseau, and Schelp~\cite{EFRS1985}, based on an argument of Bondy and Erd\H{o}s~\cite{BE}.
 For the second case, Hall's theorem is applied to ensure the embedding of a matching. 
 In the third case, it is necessary to first identify a star $K_{1,n-1}$, where Theorems \ref{bookstar} and \ref{t bookstar} play a pivotal role.

 \begin{lem} (Erd\H{o}s, Faudree, Rousseau, and Schelp \cite{EFRS1985})\label{path extension}
    Let $K_{a+b}$ be a complete graph on the vertex set
 $\{x_1,\ldots,x_a,y_1,\ldots,y_b\}$, with edges colored red or blue. Assume that there is a red path
 $x_1x_2\cdots x_a$ of length $a-1$ joining $x_1$ and $x_a$. If $a 
\geq b(c-1)+d$, then at least one of
 the following holds:
 
 \noindent
 (1) A red path of length a connects $x_1$ to $x_a$;
 
 \noindent
 (2) A blue complete subgraph $K_c$ exists;
 
 \noindent
 (3) There are $d$ vertices in $\{x_1,x_2,\ldots,x_a\}$ that are joined in blue to every vertex in
 $\{y_1,y_2,\ldots,y_b\}$.
     
 \end{lem}

\begin{lem} (Hall \cite{Hall}) \label{matching}
 Consider a complete bipartite graph $K_{a,b}$, where $a \leq b$, with parts $X = \{x_1,\ldots,x_a\}$ and $Y=\{y_1,y_2,\ldots,y_b\}$, whose edges are colored red and blue. Then, one of the following holds:

 \noindent
 (1) There exists a red matching of size $a$;

 \noindent
 (2) For some $0 \leq c \leq a-1$, there exists a blue subgraph $K_{c+1,b-c}$, where $c+1$ vertices are in $X$.   
 \end{lem}

 Moreover, the proofs of Theorems~\ref{mr} and ~\ref{mr2} require an upper bound for the Ramsey number of any graph $G$ versus $B_k$ in terms of the order and size of $G$, which is interesting of its own right.

\begin{lem} \label{upper}
    If $G$ is a  simple graph on $n$ vertices and $m$ edges, then for any positive integer $k$,
     $$r(G,B_k) \leq n+2km-\frac{2m}{n}.$$
\end{lem}
\noindent
{\bf Proof.} By induction on $n$. If $n=1$, then $m=0$ and the assertion holds trivially.
Let $v\in V(G)$ with $d(v)=\delta(G)=\delta$, and $N=\max\{r(G-v,B_k),\delta\cdot (r(G,K_{1,k})-1)+n\}$.
We first show 
 $$r(G,B_k) \leq N.$$
Assume that $K_N$ has no blue $B_k$. Since $N\geq r(G-v,B_k)$, $K_N$ contains a red $G-v$. Let $X$ be the vertices of $G-v$ which are adjacent to $v$ in $G$, and let $Y$ be the vertices of $K_N$ not in $G-v$. Then $|X|=\delta$ and $|Y|=N-n+1 \geq \delta\cdot (r(G,K_{1,k})-1)+1$. If there is a vertex of $Y$ which is red-adjacent to all vertices of $X$, then $K_N$ contains a red $G$. Otherwise, each vertex of $Y$ must be blue-adjacent to at least one vertex of $X$. Thus there exists at least $\delta\cdot (r(G,K_{1,k})-1)+1$ blue edges between $X$ and $Y$. Therefore, some vertex $w$ of $X$ has blue-degree at least $r(G,K_{1,k})$. This implies that $K_N$ contains a red $G$ or a blue $K_1+K_{1,k}=B_k$, a contradiction. Hence the above inequality holds.

 By Theorem \ref{Chv}, we have $r(G,K_{1,k})\leq r(K_n,K_{1,k})\leq nk$.
 By induction hypothesis and the inequality above, we obtain that  
 $$r(G,B_k) \leq \max\left\{n-1+2(m-\delta)k-\frac{2(m-\delta)}{n-1},\delta (nk-1)+n\right\}.$$
 The second term is at most $n+2km-2m/n$ because $\delta\le2m/n$. For the first term, $2m\ge n\delta$ gives
 \begin{align*}
 &n+2km-\frac{2m}{n}-\left(n-1+2k(m-\delta)-\frac{2(m-\delta)}{n-1}\right)\\
 &\qquad=1+2k\delta+\frac{2(m-n\delta)}{n(n-1)}
 \ge1+2k\delta-\frac{\delta}{n-1}>0.
 \end{align*}
 Therefore, we have 
 $r(G,B_k) \leq n+2 km-2m/n$
 after a direct calculation. \qed

\section{Proof of Theorem \ref{mr}} \label{section3}
Note that $B_k=K_1+K_{1,k}$. To identify a blue $B_k$ in the graph, we can find a blue star $K_{1,k}$ in the induced subgraph by $N_B(v)$ for some vertex $v$, which necessitates the following lemma.
\begin{lem} (Huang, Zhang, and Chen \cite{HZC}) \label{n+k-1}
   For integers $k \geq 1$ and $n\geq 6k^3$, let $G$ be a connected graph with $n$ vertices and at most $ n(1+1/(24k-12))$ edges. Then
$$r(G,K_{1,k})\leq n+k-1.$$
\end{lem}

Then we can establish an upper bound for the Ramsey number of sparse graph $G$ versus $B_k$ as below, which will be used to identify a subgraph $H$ of $G$ in the proof of Theorem \ref{mr}.
\begin{lem} \label{weak}
    For integers $k \geq 1$ and $n\geq 6k^3$, let $G$ be a connected graph with $n$ vertices
and at most $ n(1+1/(9k^2+23k/2+3))$ edges. Then $$r(G,B_k)\leq 2n+k-2.$$ 
\end{lem}

\noindent
{\bf Proof.} Let $N=2n+k-2$. Assume that $K_N$ is a complete graph with a red-blue
 edge-coloring such that it contains neither a red $G$ nor a blue
  $B_k$.

First, let $G_1\subseteq G$ be obtained from $G$ by recursively deleting vertices of degree 1 until the resulting graph has no vertex of degree 1. We prove, by reversing these deletions, that $K_N$ contains no red $G_1$. Let $G_1\subseteq G'\subseteq G$, and let $G''$ be obtained from $G'$ by deleting a vertex of degree 1. Assume that $K_N$ has no red $G'$ and, to the contrary, contains a red $G''$. If $v$ is the image of the support vertex of the deleted leaf, every vertex outside this red copy of $G''$ is blue-adjacent to $v$; otherwise the red $G''$ extends to a red $G'$. Since $|G| \geq 6k^3$ and $e(G)\leq n(1+1/(9k^2+23k/2+3))\leq n(1+1/(24k-12))$,
by Lemma \ref{n+k-1}, we have 
 $$|N_B(v)|\geq|K_N-G''|\geq N-(n-1)=n+k-1 \geq r(G,K_{1,k}).$$  Thus $K_N[N_B(v)]$ contains either a red $G$ or a blue $K_{1,k}$. The first alternative contradicts the global assumption, while the second, together with $v$,
 forms a blue $K_1+K_{1,k}$, a contradiction.
Thus $K_N$ contains no red $G_1$.

 Second, let $G_0$ be the graph obtained from $G_1$ by shortening any suspended path with at least $3k+3$ vertices to one with $3k+2$ vertices.
 In this process, let $G''$ be a graph obtained from $G'$ by shortening any suspended path with at least $3k+3$ vertices by a vertex.
We will prove that $K_N$ contains no red $G_0$
 by showing if $K_N$ has no red $G'$, then $K_N$ contains no $G''$ too. Suppose to the contrary that $K_N$ contains a red $G''$. 
 Since $|K_N-G''|\geq N-(n-1)=n+k-1\geq n$, $K_N-G''$ contains at least one blue edge $y_1y_2$, for otherwise it contains a red $K_n$ and so a red $G$. 
 Let $P=x_1x_2\cdots x_a$ be the suspended path in $G''$, where $a\geq 3k+2=2((k+2)-1)+k$. Set $b=2$, $c=k+2$ and $d=k$. Then we have $a \geq b(c -1)+d$. The first alternative of Lemma~\ref{path extension} would extend the red $G''$ to a red $G'$, contrary to the choice of $G'$. Hence, applying the lemma to the red $P$ and $\{y_1,y_2\}$, we either have a blue  $K_{k+2}$ or
there exist $k$ vertices among $\{x_1,x_2,\ldots,x_a\}$ such that every edge between these $k$ vertices and the vertices $y_1,y_2$ is blue. In the first case, we obtain a blue $K_{k+2}$,
 and in the second case, we obtain a blue $K_2+\overline{K}_k$. In both cases, there
 exists a blue $B_k$, again a contradiction. Thus $K_N$ contains no red $G_0$.

Assume $G_0$ has $\ell$ vertices and $m$ edges, then $m\leq \ell+n/(9k^2+23k/2+3)$. Moreover, since $K_N$ contains no red $G_0$ and $G_0$ is connected, we must have $\ell \geq 3$ and $m\geq \ell-1$.
If $G_0$ is a cycle, then $\ell\leq3k+2$. For $k\geq2$, Lemma~\ref{upper} gives
$$r(C_\ell,B_k)\leq(2k+1)\ell-2\leq6k^2+7k\leq2n+k-2=N,$$
where the last inequality follows from $n\geq6k^3$. For $k=1$, we have $3\leq\ell\leq5$ and $N=2n-1\geq11$; the classical equality $r(C_3,K_3)=6$ and the equalities $r(C_4,K_3)=7$ and $r(C_5,K_3)=9$ (the latter two follow from~\cite{Shi 2010}) again give $r(C_\ell,B_1)\leq N$.
In either case, the absence of a blue $B_k$ forces a red $G_0$, a contradiction. Thus $G_0$ is not a cycle.
Recall that $G_0$ has no vertices of degree 1 and no suspended path with more than $3k+2$ vertices. Since $G_0$ is connected and is not a cycle, Lemma~\ref{old} gives
$$\frac{\ell}{2(3k+2)}-\frac{3n}{2(9k^2+23k/2+3)}\leq 0.$$
 Therefore, $\ell \leq 3(3k+2)n/(9k^2+23k/2+3)$. Hence, by Lemma \ref{upper},
\begin{align*}
    r(G_0,B_k)\leq &\ell+2km-\frac{2m}{\ell}\\
    \leq& \ell +2k\left(\ell+\frac{n}{9k^2+23k/2+3}\right)-\frac{2(\ell-1)}{\ell}\\
    \leq & (2k+1)\ell+\frac{2k}{9k^2+23k/2+3}n-1 \\
    \leq& \frac{18k^2+23k+6}{9k^2+23k/2+3}n-1\\
    =&2n-1 \\
    \leq &N.
\end{align*}
This is to say that $K_N$ has a red $G_0$, a contradiction, and this completes the proof. \qed

\vskip 5mm
We now begin to  prove Theorem \ref{mr}.

By Theorem \ref{lowerbound}, it suffices to prove the upper bound. Let $N=2n-1$. Assume that $K_N$ has a red-blue
 edge-coloring such that it contains neither a red $G$ nor a blue $B_k$. 

 If $k=1$, the assertion follows directly from Lemma~\ref{weak}, since the hypotheses of Theorem~\ref{mr} imply those of that lemma. Hence assume $k\ge2$. By Theorem \ref{bookstar}, the case where $G$ is a star is already settled. Thus we assume $G$ is not a star, and then 
 by Lemma \ref{trichotomy}, we need to discuss three cases separately.

\vskip 2mm 
\noindent 
{\textbf{Case 1.}} $G$ has a suspended path with $\lceil(k-1)/2 \rceil+3k+2$ vertices.
\vskip 2mm

 Let $H$ be the graph obtained from $G$ by shortening the suspended path by $\lceil(k-1)/2 \rceil$ vertices. Since $n\geq 34k^3$, we have that
 \begin{align*}
 |H| &=n-\left \lceil\frac{k-1}{2} \right \rceil \geq  6k^3,\\
   e(H) &\leq n\left(1+\frac{1}{119k^2+62k}\right)-\left \lceil\frac{k-1}{2}\right\rceil  
   \leq \left(n-\left\lceil\frac{k-1}{2} \right\rceil \right) \left(1+\frac{1}{9k^2+23k/2+3}\right).
 \end{align*}
 Then by Lemma \ref{weak}, $$r(H,B_k)\leq 2(n-\lceil(k-1)/2 \rceil)+k-2\leq 2n-1=N.$$ 
 Therefore, $K_N$ contains a red $H$. Let $H'$ be the subgraph of $K_N[R]$ which can
 be obtained from $H$ by lengthening the suspended path as much as possible, up to
 $\lceil(k-1)/2 \rceil$ vertices. Let  $P=x_1x_2\cdots x_a$ be the path, where $a\geq 3k+2=2((k+2)-1)+k$. If $a=\lceil(k-1)/2 \rceil+3k+2$, then we get a red $G$, a contradiction. If not, then $|K_N-H'|\geq N-(n-1)=n$, so there must be a blue edge $y_1y_2$ in $K_N-H'$; otherwise this set contains a red $K_n$, and hence a red $G$. Set $b=2$, $c=k+2$ and $d=k$.  So we have $a \geq b(c -1)+d$. Applying Lemma \ref{path extension} on the red path $P$ and $\{y_1,y_2\}$, the first alternative contradicts the maximality of $P$. Hence we either have a blue $K_{k+2}$, which contains a blue $B_k$, or
there exist $k$ vertices among $\{x_1,x_2,\ldots,x_a\}$ such that every edge between these $k$ vertices and the vertices $\{y_1,y_2\}$ is blue which implies there is a blue $K_2+\overline{K}_k=B_k$, a contradiction. 

 \vskip 2mm
  \noindent  {\textbf{Case 2.}} $G$ has a matching consisting of $2k-1$ end-edges.
 \vskip 2mm

Let $H$ be the graph obtained from $G$ by deleting $2k-1$ end-vertices of this matching. Since $n\geq 34k^3$, we have that
 \begin{align*}
 |H| &=n-(2k-1) \geq 6k^3,\\
   e(H) &\leq n \left(1+\frac{1}{119k^2+62k}\right)-(2k-1) \leq (n-(2k-1))\left(1+\frac{1}{9k^2+23k/2+3}\right).
 \end{align*}
By Lemma \ref{weak}, 
$$r(H,B_k)\leq 2(n-(2k-1))+k-2 =2n-3k<2n-1= N.$$ 
Thus $K_N$ contains a red $H$. 
Let $X$ be the set of $2k-1$ vertices of $H$ which were adjacent to the end-vertices deleted from $G$ and $Y=V(K_N)-V(H)$. By Lemma \ref{matching}, there is either a red matching that covers $X$, or a blue $K_{c+1,|Y|-c}$, where $0 \leq c \leq |X|-1$. 
If the former holds, then $K_N$ would contain a red $G$, a contradiction. If the latter holds, let $X'$ be the vertices of $X$ and $Y'$ be the vertices of $Y$ contained in the blue $K_{c+1,|Y|-c}$. Since $|Y'|=|Y|-c\geq 2n-1-(n-|X|)-(|X|-1)=n$, there must be one blue edge in $K_N[Y']$. Thus, if $c+1\geq k$, then the blue $K_2$ in $K_N[Y']$ along with $k$ vertices of $X'$ gives a blue $K_2+\overline{K}_k=B_k$, a contradiction.
If  $c<k-1$, then by Lemma \ref{n+k-1},
  $$|Y'|=|Y|-c> 2n-1-(n-(2k-1))-(k-1)= n+k-1\geq r(G,K_{1,k}).$$
Therefore, $K_N[Y']$ contains a blue $K_{1,k}$, which along with a vertex in $X'$ gives a blue $K_1+K_{1,k}=B_k$, a final contradiction.

 \vskip 2mm
\noindent {\textbf{Case 3.}} 
$G$ has neither a suspended path of $\lceil(k-1)/2 \rceil+3k+2$ vertices, nor a matching formed by $2k-1$ end-edges.
\vskip 2mm

For convenience, we let $e(G)=n+\ell$, $q=\lceil(k-1)/2 \rceil+3k+2$, and $s=2k-1$, where $\ell\leq n/(119k^2+62k)$.

Let $G'$ be a graph obtained from $G$ by removing all vertices of degree 1. Clearly, $G'$ is connected. Because $n\geq 34k^3$
by Lemma \ref{trichotomy}, we have
\begin{align*}
   |G'| \leq \gamma &=(q-2)(2s+3\ell-2)+1\\
   &\leq \frac{7k}{2} \left(4k-4+\frac{3n}{119k^2+62k}\right)+1 \\
   &=14k^2-14k+1+\frac{21k}{2(119k^2+62k)}n \\
   &\leq \frac{49k}{119k^2+62k}n+\frac{21k}{2(119k^2+62k)}n\\
   & = \frac{119k}{2(119k^2+62k)}n.
\end{align*}
Moreover,
 $$|G'|-1 \leq e(G') \leq |G'|+\frac{n}{119k^2+62k}\leq  \frac{119k+2}{2(119k^2+62k)}n.$$
By Lemma \ref{trichotomy},  $G$ has a vertex adjacent to at least 
$$\left\lceil \frac{n-\gamma}{2k-2} \right\rceil$$ 
vertices of degree 1.
Denote this vertex by $v$. By Theorem \ref{bookstar}, there exists a red star $K_{1,n-1}$ in $K_N$, with the center
 vertex denoted by $x$. By Lemma \ref{upper}, if $|G'| \geq 2$, then
\begin{align*}
    r(G',B_k)&\leq |G'|+2ke(G')-\frac{2e(G')}{|G'|}\\
    &\leq  \frac{119k}{2(119k^2+62k)}n+\frac{2k(119k+2)}{2(119k^2+62k)}n -\frac{2(|G'|-1)}{|G'|}\\
    &\leq \frac{119k^2+123k/2}{119k^2+62k}n-1 \\
    &< n-1.
\end{align*}
Here we used $e(G')\geq |G'|-1$ and $|G'|\geq2$, which imply $-2e(G')/|G'|\leq-1$.
Thus $K_N[N_R(x)]$ contains a red $G'$. Moreover, this holds trivially if $|G'| = 1$. Replacing $v$ in this $G'$ by $x$ gives another copy of red $G'$. 
 Next, based
on the embedding of $G'$, we apply a greedy algorithm to
embed $G''$ into
 the red subgraph of $K_N$, where $G''$ is obtained from $G$ by removing all degree-1 vertices adjacent to $v$. Clearly, $G''$ contains $G'$ as a subgraph and
 $$|G''| \leq n-\left\lceil \frac{n-\gamma}{2k-2}\right\rceil.$$
If at some step the embedding cannot proceed, the image $z$ of the relevant support vertex has no unused red neighbor and hence is adjacent to at most $|G''|-2$ vertices by red edges. In other words, the
 number of blue edges incident to the vertex $z$ is at least $N-|G''|+1$.
Because $n\geq 34k^3$,
we have $n-\gamma \geq (k-1)(2k-2)$.
Then by Lemma \ref{n+k-1}, we have $$|N_B(z)| \geq N-|G''|+1 \geq 2n-1-n+\left\lceil \frac{n-\gamma}{2k-2}\right\rceil +1\geq n+k-1 \geq r(G,K_{1,k}).$$ Then in the blue neighborhood of the vertex
 $z$, there is a blue $K_{1,k}$, which along with $z$ gives a blue $B_k$, a contradiction.

 Thus, $G''$ can always be embedded into the red subgraph of $K_N$. Next, we only need
 to embed the degree-1 vertices adjacent to $v$, and the entire graph $G$ can be embedded
 into the red subgraph of $K_N$. This is feasible because the vertex $x$ into which $v$ is
 embedded is adjacent to at least $n-1$ red edges, so there are enough vertices to embed the degree-1 vertices of $v$. This completes the proof of the theorem. \qed

\section{Proof of Theorem \ref{t bookstar}} \label{section4}
The following lemma is a classical result due to Andr{\'a}sfai, Erd{\H{o}}s, and S{\'o}s in 1974 concerning chromatic number, maximal clique, and minimum degree. From this lemma, we get that when a  graph contains no $K_3$ and 
its minimum degree is large, it must be a bipartite graph.
\begin{lem}(Andr{\'a}sfai, Erd{\H{o}}s, and S{\'o}s \cite{AES})\label{F123}
    Let $r\geq  3$. For any graph $F$, at most two of the following
 properties can hold: (1) $F$ contains no $K_r$; (2) $\delta(F)>\frac{3r-7}{3r-4}|F|$; (3) $\chi(F)\geq r$.
\end{lem}

We now proceed to prove Theorem \ref{t bookstar}.

By Theorem \ref{lowerbound}, it suffices to prove the upper bound. 
We use induction on $t$. When $t=1$, it is verified
 by Theorem \ref{bookstar}. Assuming the theorem holds for $t -1$, we now proceed to consider the case where $k\geq 1$ and $t \geq 2$.

Let $N=2n+t-2$. Given a red-blue
 edge-coloring of  $K_N$, assume that it contains neither a red $K_{1,n-1}$ nor a blue $tB_k$. 

By the induction hypothesis, we have
$$r(K_{1,n-1},(t-1)B_k)= 2n+t-3<N.$$
Thus $K_N$ contains a blue $(t-1)B_k$. Let $A$ be the
 vertex set of this $(t-1)B_k$ and $S=V(K_N)-A$.   
 Clearly, $K_N[S]$ has no blue $B_k$. Since $K_N$ has no red $K_{1,n-1}$,  we have 
 $$|N_B(v)| \geq N-1-(n-2)=n+t-1\text{ for any $v\in V(K_N)$}.$$

Let $F$ be the subgraph induced by all blue edges in $K_N[S]$.
\begin{claim} \label{1}
$\delta(F)\geq n-(t-1)(k+1)$.
\end{claim}
\begin{proof}
Note that $|N_B(v)| \geq n+t-1 $ for any $v\in V(K_N)$, we have
$$\delta(F) \geq n+t-1-|A|= n+t-1-(t-1)(k+2)=n-(t-1)(k+1),$$    
and so the assertion follows.   
\end{proof}
\begin{claim} \label{no K_3}
  $F$ contains no $K_3$.  
\end{claim}
\begin{proof}
Suppose to the contrary that $F$ contains a $K_3$ with vertex set
 $\{x,y,z\}$. Set $$X=N(x)-\{y,z\}, ~Y=N(y)-(N(x)\cup \{x\}), ~Z=N(z)-(N(x)\cup N(y)).$$
Since $K_N$ contains no blue $tB_k$, $F$ contains no $B_k$. Then by Claim \ref{1},
\begin{align*}
    |X| &\geq n-(t-1)(k+1)-2,\\
    |Y| &\geq n-(t-1)(k+1)-2-(k-2)=n-(t-1)(k+1)-k,\\
    |Z| &\geq n-(t-1)(k+1)-2-2(k-2)=n-(t-1)(k+1)-2k+2.
\end{align*}
Since $X\cup Y\cup Z \subseteq S$, and $X$, $Y$ and $Z$ are pairwise disjoint, we have
\begin{align*}
  2n-(t-1)(k+1)-1  =|S| &\geq | \{x,y,z\} \cup X\cup Y\cup Z| \\
  & \geq 3+3n-3(t-1)(k+1)-3k.
\end{align*}
This gives $n\leq 2tk+2t+k-6$, which contradicts the assumption $n\geq 3tk+3t-5$.
\end{proof}
By Claim \ref{1}, we have $\delta(F)\ge n-(t-1)(k+1)$. Since $n\geq 3tk+3t-5$ and $|F|=N-|A|=2n-(t-1)(k+1)-1$, we have
$\delta(F)>\frac{2}{5}(2n-(t-1)(k+1)-1)=\frac{2}{5}|F|.$
By Claim \ref{no K_3}, $F$ has no $K_3$.  Thus, applying Lemma \ref{F123} on $F$ for $r=3$, we have 
$\chi(F)\le 2$.
Note that $|F| \geq n$, we have $\delta(F)>\frac{2}{5}n$ and so $\chi(F)=2$. Let $(S_1, S_2)$ be a bipartition of $F$. Then by Claim \ref{1}, $|S_i|\geq \delta(F)\ge n-(t-1)(k+1)$ for $i=1,2$. Moreover, since $K_N[S_i]$ is a red complete graph for $i=1,2$, and $K_N$ contains no $K_{1,n-1}$, $|S_i|\le n-1$. Thus, we have  
$$n-(t-1)(k+1)\leq |S_i|\leq n-1\text{ for $i=1,2$.}$$

Recall that $|N_B(v)| \geq n+t-1$ for all $v\in V(K_N)$, every vertex in 
$A$ must have many blue neighbors in $S$. In fact, for each $a\in A$, 
$$|N_B(a)\cap S| \geq n+t-1-(|A|-1)=n-(t-1)(k+1)+1.$$ 
Since $K_N$ has no $tB_k$, each $B_k$ in $A$ cannot form a blue $2B_k$ with some vertices in $S$. Thus we have the following claim.

\begin{claim} \label{A}
   For each $B_k$ in $A$, there is at most one vertex whose blue neighbors are at least 2 in both $S_1$ and $S_2$.
\end{claim}   
\begin{proof} Let $a_1,a_2$ be any two vertices of some $B_k$ in $A$. 

If $|N_B(a_i)\cap S_{3-i}|\ge (t-1)(k+1)+k$ and $b_i\in N_B(a_i)\cap S_i$ for some $i\in\{1,2\}$, then by Claim \ref{1}, the blue common neighbors of $a_i$ and $b_i$ in $S_{3-i}$ are at least
\begin{align*}
  &~ |N_B(a_i)\cap S_{3-i}|+|N_B(b_i)\cap S_{3-i}| -|S_{3-i}| \\
   \geq &~(t-1)(k+1)+k + n-(t-1)(k+1)-(n-1)
   = k+1. \tag{2} \label{eq:degree_bound}
   \end{align*}

If $|N_B(a_i)\cap S_j|\ge (t-1)(k+1)+k$ for $i,j=1,2$, say $b_i\in N_B(a_i)\cap S_i$, then by \eqref{eq:degree_bound}, $a_i$ and $b_i$ have at least $k+1$ blue common neighbors in $S_{3-i}$ for $i=1,2$. This implies that $K_N[\{a_1,a_2\}\cup S]$ contains a blue $2B_k$, and hence $K_N$ has a blue $tB_k$, a contradiction. Thus, at most one of $a_1,a_2$ satisfies
$$|N_B(a_i)\cap S_j| \geq (t-1)(k+1)+k\text{ for $j=1,2$}.$$

If a vertex $v\in A$ satisfies the condition above, we call $v$ a special vertex. Moreover,  because $|N_B(v)\cap S| \geq n-(t-1)(k+1)+1$ for any $v\in A$ and $n\geq 3tk+3t-5$, there is no vertex $v\in A$ such that $|N_B(v)\cap S_j| < (t-1)(k+1)+k\text{ for $j=1,2$}.$
Thus, if $v\in A$ is not a special vertex, then $|N_B(v)\cap S_j| \geq (t-1)(k+1)+k$ for some $j\in \{1,2\}$.

Now, suppose that there are two vertices $a_1,a_2$  such that $|N_B(a_i)\cap S_j| \geq 2$ for $i,j=1,2$. By the arguments above, we distinguish the following two cases.

\vskip2mm
\noindent 
{\textbf{Case 1.}} Both $a_1$ and $a_2$ are not special vertices.
\vskip2mm
By symmetry, we only need to consider the following two subcases.

If $|N_B(a_i)\cap S_{3-i}|\geq (t-1)(k+1)+k$ for $i=1,2$, select $b_i \in N_B(a_i)\cap S_i$. By \eqref{eq:degree_bound},  $a_i$ and $b_i$ have at least $k+1$ common blue neighbors in $S_{3-i}$ for $i=1,2$.

 If $|N_B(a_i)\cap S_1| \geq (t-1)(k+1)+k$ for $i=1,2$, then since $a_1$ and $a_2$ are not special vertices, we have $|N_B(a_i)\cap S_2| <(t-1)(k+1)+k$.  
   Select distinct vertices $b_2^{i} \in N_B(a_i)\cap S_2$ for $i=1,2$; this is possible because each $a_i$ has at least two blue neighbors in $S_2$. Since $n\geq 3tk+3t-5$, the number of common blue neighbors of $a_i$ and $b_2^i$ in $S_1$ is at least
   \begin{align*}
  & |N_B(a_i)\cap S_1|+|N_B(b_2^{i})\cap S_1| -|S_1| \\
   \geq &|N_B(a_i)\cap S|-|N_B(a_i)\cap S_2| + |N_B(b_2^{i})\cap S_1| -|S_1|\\
   \geq& 2(n-(t-1)(k+1))-((t-1)(k+1)+k-1)-(n-1)\\
   =& n-3(t-1)(k+1)-k+2\\
   \geq& 2k.
   \end{align*}
In the first subcase, the $k+1$ common neighbors allow the two page sets to avoid the other selected spine vertex. In the second subcase, each pair has at least $2k$ common neighbors, so two disjoint sets of $k$ pages can be chosen greedily. Thus in both subcases, $K_N[\{a_1,a_2\}\cup S]$ contains a blue $2B_k$, and then $K_N$ has a blue $tB_k$, a contradiction. 
\vskip2mm
\noindent 
{\textbf{Case 2.}} $a_1$ is a special vertex and $a_2$ is not.
\vskip2mm
In this case, $|N_B(a_2)\cap S_i| \geq (t-1)(k+1)+k$ for some $i\in \{1,2\}$. By symmetry of $S_1$ and $S_2$, we may assume $|N_B(a_2)\cap S_1| \geq (t-1)(k+1)+k$.
Select $b_i \in N_B(a_i)\cap S_i$ for $i=1,2$.  
Then by \eqref{eq:degree_bound},  $a_i$ and $b_i$ have at least $k+1$ common blue neighbors in $S_{3-i}$ for $i=1,2$. Choose the page set of each book to avoid the other selected spine vertex.
Thus, $K_N[\{a_1,a_2\}\cup S]$ has a blue $2B_k$, and so $K_N$ has a blue $tB_k$, a contradiction. 
\end{proof}

For $i=1,2$, we define 
$$A_i=\{a\in A : |N_B(a)\cap S_i|\leq 1\}.$$
Since $|N_B(a)\cap S| \geq n-(t-1)(k+1)+1$ for $a\in A$, then $A_1$ and $A_2$ are disjoint. By Claim \ref{A}, we have $|A_1|+|A_2| \geq |A|-(t-1).$ Then $$|S_1 \cup A_1|+ |S_2 \cup A_2| \geq |S|+|A|-(t-1) \geq N-(t-1)=2n-1.$$
Thus either $|S_1 \cup A_1| \geq n$ or $|S_2 \cup A_2| \geq n$. By symmetry, assume that $|S_1 \cup A_1| \geq n$. By the definition of $A_1$, each vertex of $A_1$ has at most one blue neighbor in $S_1$. Moreover, because $n\geq 3tk+3t-5$,
   $$|A_1|\leq |A| = (t-1)(k+2)<n- (t-1)(k+1)\leq |S_1|.$$ 
Each vertex of $A_1$ has at most one blue neighbor in $S_1$, so at most $|A_1|$ vertices of $S_1$ have a blue neighbor in $A_1$. Since $|A_1|<|S_1|$, choose $v\in S_1$ with no blue neighbor in $A_1$. As $K_N[S_1]$ is red complete, $v$ is red-adjacent to every vertex of $(S_1 \cup A_1) \setminus \{v\}$. Thus $K_N[S_1 \cup A_1]$ contains a red $K_{1,n-1}$.

The proof of Theorem \ref{t bookstar} is complete. \qed

\section{Proof of Theorem \ref{mr2}} \label{section5}
\begin{lem} (Burr, Erd{\H{o}}s, and Spencer \cite{BES1975})
    Let $G$ and $H$ be graphs. For any positive integer $t$, we have 
    $r(G, tH)\leq r(G,H)+(t-1)|H|$.
\end{lem}

Then, from the exact value of 
$r(G,B_k)$ given by Theorem \ref{mr} and the upper bound provided by Lemma \ref{upper}, we obtain the following corollaries.

\begin{cor} \label{1upper tB_k}
    Let $k,t$ be positive integers and $n\geq 34k^3$. If $G$ is a connected graph with $n$ vertices
and at most $ n(1+1/(119 k^2 + 62 k))$ edges, then
$r(G,tB_k)\leq 2n-1+(t-1)(k+2)$.
\end{cor}

\begin{cor} \label{3upper tB_k}
    If $G$ is a  simple graph on $n$ vertices and $m$ edges, then for any positive integers $k,t$,
     $$r(G,tB_k) \leq n+2m k-\frac{2m}{n}+(t-1)(k+2).$$
\end{cor}

Because $tB_k$ is a subgraph of $tK_2+\overline{K}_{tk}$, it is essential to first identify a blue $tK_2$. Thus we need the following two lemmas, which yield that $r(G,tK_2)=n+t-1$ for some graph $G$ on $n$ vertices and $t\geq 2$.
 \begin{lem} (Chv\'atal and Harary \cite{Chvatal1972}) \label{2K_2}
For any graph $G$ on $n$ vertices that contains no
 isolated vertices and $G$ is not complete, 
 $r(G,2K_2)=n+1$.
    
\end{lem}

\begin{lem} (Zhang and Chen \cite{ZC}) \label{tK2}
Let $t\geq 3$ be an integer and $n \geq 3(4t-5)(2t-3)$.
     If $G$ is a connected graph with $n$ vertices and  at most $n+n^2/(4t-5)-2$ edges, then
 $r(G,tK_2) = n+t-1$.
     
\end{lem}

\vskip 2mm
We now begin to  prove Theorem \ref{mr2}.

 By Theorem \ref{lowerbound}, it suffices to prove the upper bound. We use induction on $t$. When $t=1$, it is verified
 by Theorem \ref{mr}. Assuming the theorem holds for $t -1$, we now proceed to consider the case where $k\geq 2$ and $t \geq 2$.

 To apply the inductive method, it is necessary to ensure that the lower bound on $n$ and the upper bound on $e(G)$ are
 compatible with induction. This can be derived from the facts that the lower bound on $n$ is increasing in $t$ and the upper bound on $e(G)$ is decreasing in $t$.
 
 Let $N=2n+t-2$. Given a red-blue edge-coloring of $K_N$, assume that $K_N$ contains neither a red $G$ nor a blue $tB_k$.
 We record that the hypotheses needed to use Lemmas~\ref{2K_2} and~\ref{tK2} are satisfied throughout the proof. The graph $G$ is connected, has no isolated vertices, and is noncomplete under the stated size condition. 
 When $t\ge3$, the hypotheses $n\ge111t^3k^3$ and $e(G)\le n(1+1/(127t^2k^2+79t^2k))$ imply both $n\ge3(4t-5)(2t-3)$ and $e(G)\le n+n^2/(4t-5)-2$.
 Thus, on every set of at least $n+t-1$ vertices, the absence of a red $G$ forces a blue $tK_2$.
 By Theorem \ref{t bookstar}, the case where $G$ is a star is already settled. Thus we assume $G$ is not a star, and then 
 our argument divides into three cases.

\vskip 2mm
\noindent {\textbf{Case 1.}} $G$ has a suspended path with at least $2t(tk+2t-1)+tk+\lceil(tk+t-k-1)/2\rceil$ vertices.
\vskip 2mm
Let $H$ be the graph obtained from $G$ by shortening the suspended path by $\lceil(tk+t-k-1)/2\rceil$ vertices. Since $n\geq 111t^3k^3$, we have 
 \begin{align*}
 |H| &=n-\lceil(tk+t-k-1)/2\rceil \geq 34k^3, \\
   e(H) &\leq n\left(1+\frac{1}{127 t^2 k^2 + 79 t^2 k}\right)-\lceil(tk+t-k-1)/2\rceil \\
   &\leq \left(n-\lceil(tk+t-k-1)/2\rceil \right) \left(1+\frac{1}{119k^2+62k}\right).
 \end{align*}
 Then by Corollary \ref{1upper tB_k}, $$r(H,tB_k)\leq 2(n-\lceil (tk+t-k-1)/2\rceil)-1+(t-1)(k+2)\leq 2n+t-2=N.$$ Therefore $K_N$ contains a red $H$. Let $H'$ be the subgraph of $K_N[R]$ which can
 be obtained from $H$ by lengthening the suspended path as much as possible (up to
 $\lceil(tk+t-k-1)/2\rceil$ vertices). If $H'$ is $G$, the proof of this case is complete. If not, since $|K_N-H'|\geq N-(n-1)=n+t-1$, there must be a blue $tK_2$ in $K_N-H'$ by Lemmas \ref{2K_2} and \ref{tK2}. Let the vertices of the suspended path $P$ be relabeled as $x_1,x_2,\ldots,x_a$, where $a\geq 2t(t(k+2)-1)+tk$. Let the vertices of the blue $tK_2$ be relabeled as $y_1,y_2,\ldots,y_b$, and then let $b=2t$. Also, let $c=t(k+2)$ and $d=tk$. So $a \geq b(c -1)+d$. By applying Lemma \ref{path extension}, its first alternative contradicts the maximality of $P$; hence we either have a blue subgraph $K_{t(k+2)}$, or
there exist $tk$ vertices among $\{x_1,x_2,\ldots,x_a\}$ such that every edge between these $tk$ vertices and the vertices $\{y_1,y_2,\ldots,y_b\}$ is blue which implies there is a blue subgraph $tK_2+\overline{K}_{tk}$. In both cases, there
 exists a blue subgraph $tB_k$, a contradiction 

 \vskip 2mm
  \noindent  {\textbf{Case 2.}} $G$ has a matching consisting of $2tk+t-2$ end-edges.
 \vskip 2mm
By the induction hypothesis, we have
 $r(G,(t-1)B_k)\leq 2n+t-3<N$,
 which implies that the graph $K_N$ contains a blue subgraph $(t-1)B_k$. We denote the
 vertex set of this subgraph by $A$.

Let $H$ be the graph obtained from $G$ by deleting $2tk+t-2$ end-vertices of this matching. Because $n\geq 111t^3k^3 $,
\begin{align*}
|H| &=n-2tk-t+2 \geq 34k^3, \\
e(H) & \leq  n\left(1+\frac{1}{127 t^2 k^2 + 79 t^2 k}\right)-(2tk+t-2)\\
& \leq  (n-2tk-t+2)\left(1+\frac{1}{119k^2+62k}\right).
\end{align*}
By Theorem \ref{mr}, we get 
\begin{align*}
   r(H,B_k) \leq & 2(n-2tk-t+2)-1\\
   =& 2n-4tk-2t+3\\
   \leq & 2n+t-2-(t-1)(k+2)\\
   =& N-|A|,
\end{align*}
which implies that $K_N-A$ contains a red $H$.
Let $X$ be the set of $2tk+t-2$ vertices of $H$ which were adjacent to the end-vertices deleted from $G$ and $Y=V(K_N)-V(H)$. By Lemma \ref{matching}, there is either a red matching that covers $X$, or a blue $K_{c+1,|Y|-c}$, where $0 \leq c \leq |X|-1$. 
If the former holds, then $K_N$ would contain a red subgraph
 isomorphic to $G$, leading to a contradiction. If the latter holds, let $X'$ be the vertices of $X$ and $Y'$ be the vertices of $Y$ contained in the blue $K_{c+1,|Y|-c}$. Since $|Y'|=|Y|-c\geq 2n+t-2-(n-|X|)-(|X|-1)=n+t-1$, there must be a blue $tK_2$ by Lemmas \ref{2K_2} and \ref{tK2} in $K_N[Y']$. Thus, if $c+1\geq tk$, then the blue $tK_2$ in $K_N[Y']$ along with $tk$ vertices of $X'$ gives a blue $tK_2+\overline{K}_{tk}$, which contains $tB_k$ as a subgraph.
  Next consider the case  $c<tk-1$. Thus by Lemma \ref{n+k-1},
  \begin{align*}
      |Y'|&> 2n+t-2-(n-(2tk+t-2))-(tk-1)\\
      &=n+k-1+(t-1)(k+2)\\
      &=n+k-1+|A|.
  \end{align*}
Therefore $K_N[Y'-A]$ contains a blue $K_{1,k}$, which along with a vertex in $X'$ gives a blue $K_1+K_{1,k}=B_k$. Combining it with the blue $(t-1)B_k$ from $A$ would yield a blue $tB_k$. Thus for each of the two possibilities we 
reach a contradiction.

 \vskip 2mm
\noindent {\textbf{Case 3.}} 
$G$ has neither a suspended path of $2t(tk+2t-1)+tk+\lceil(tk+t-k-1)/2\rceil$ vertices, nor a matching formed by $2tk+t-2$ end-edges.
\vskip 2mm
For convenience, we let $e(G)=n+\ell$, $q=2t(tk+2t-1)+tk+\lceil(tk+t-k-1)/2\rceil$, and $s=2tk+t-2$, where $\ell\leq n/(127 t^2 k^2 + 79 t^2 k)$.

After removing all vertices of degree 1 from the graph $G$, we
 obtain a graph denoted by $G'$. Clearly, $G'$ is a connected graph. 
Because $n\ge111t^3k^3$, by Lemma \ref{trichotomy}, we have
\begin{align*}
   |G'| &\leq \gamma=(q-2)(2s+3\ell-2)+1\\
   &\leq (q-2)(2s-2)+1+\frac{3(q-2)}{127 t^2 k^2 + 79 t^2 k}n \\
    &\leq \frac{49t^2k}{127 t^2 k^2 + 79 t^2 k}n+\frac{6t^2k+12t^2+9tk/2-9t/2-3k/2-6}{127 t^2 k^2 + 79 t^2 k}n\\
   & = \frac{55t^2k+12t^2+9tk/2-9t/2-3k/2-6}{127 t^2 k^2 + 79 t^2 k}n.
\end{align*}
Moreover,
\begin{align*}
|G'|-1 \leq e(G') \leq & |G'|+\frac{n}{127 t^2 k^2 + 79 t^2 k}\\
\leq &  \frac{55t^2k+12t^2+9tk/2-9t/2-3k/2-5}{127 t^2 k^2 + 79 t^2 k}n.    
\end{align*}
By Lemma \ref{trichotomy},  $G$ has a vertex adjacent to at least 
$$\left\lceil \frac{n-\gamma}{2tk+t-3} \right\rceil$$ 
vertices of degree 1.
Denote this vertex by $v$. By Theorem \ref{t bookstar}, there exists a red star $K_{1,n-1}$ in $K_N$, with the center
 vertex denoted by $x$. By Corollary \ref{3upper tB_k}, if $|G'| \geq 2$, then
\begin{align*}
r(G',tB_k)&\leq |G'|+2ke(G')-\frac{2e(G')}{|G'|}+(t-1)(k+2)\\
    &\leq \frac{110t^2k^2 + 79t^2k + 12t^2 + 9tk^2-4tk-4t}{127 t^2 k^2 + 79 t^2 k}n-1\\
    &<  n-1.
\end{align*}
Thus $K_N[N_R(x)]$ contains a red $G'$. Moreover, this holds trivially if $|G'| = 1$. Replacing $v$ in this $G'$ by $x$ gives another copy of red $G'$. Every vertex of $G''-G'$ is an original leaf of $G$ whose support lies in $G'$. Embed these leaves one at a time, always after their support has been embedded. In this way we greedily embed $G''$ into
 the red subgraph of $K_N$, where $G''$ is obtained by removing all degree-1 vertices adjacent to $v$ in the graph $G$. Clearly, $G''$ contains $G'$ as a subgraph and
 $$|G''| \leq n-\left\lceil \frac{n-\gamma}{2tk+t-3}\right\rceil.$$
If at some step the embedding cannot proceed, the image $z$ of the relevant support vertex has no unused red neighbor and hence is adjacent to at most $|G''|-2$ vertices by red edges. In other words, the
 number of blue edges incident to vertex $z$ is at least $N-|G''|+1$. 
Since $n\ge111t^3k^3$, we have $n-\gamma \geq (2tk+t-3)(tk+t-2)$, then
 \begin{align*}
    N-|G''|+1 &\geq 2n+t-2-n+\left\lceil \frac{n-\gamma}{2tk+t-3}\right\rceil +1\\
    &=n+t-1+\left\lceil \frac{n-\gamma}{2tk+t-3}\right\rceil \\
   &\geq  n+t-1+(tk+t-2)\\
    & = n+k-1+(t-1)(k+2)\\
    & \geq r(G,K_{1,k})+|(t-1)B_k|.
 \end{align*}
By the induction hypothesis, we have
 $r(G,(t-1)B_k)\leq 2n+t-3=N-1$,
 which implies that the graph $K_N\setminus \{z\}$ contains a blue subgraph $(t-1)B_k$. We denote the
 vertex set of this subgraph by $A$. In the blue neighborhood of vertex
 $z$, there are at least $n+k-1$ vertices that do not belong to the set $A$. By Lemma \ref{n+k-1}, in the graph induced by these vertices, there exists either a red subgraph
 isomorphic to $G$, or a blue $K_{1,k}$, which, together with vertex $z$ and the vertices in $A$,
 induces a blue subgraph containing $tB_k$. Both cases lead to a contradiction.

 Thus, $G''$ can always be embedded into the red subgraph of $K_N$. Next, we only need
 to embed the degree-1 vertices adjacent to $v$, and the entire graph $G$ can be embedded
 into the red subgraph of $K_N$. This is feasible because the vertex $x$ into which $v$ is
 embedded is adjacent to at least $n-1$ red edges, so there are enough vertices to embed the degree-1 vertices of $v$. This completes the proof of Theorem \ref{mr2}. \qed

\section{Concluding Remark}
The restrictions on the range of $n$ and the size of $G$ are necessitated by the third cases in the proofs  of Theorems \ref{mr} and \ref{mr2}. Indeed, there exists an inherent trade-off between the lower bound on $n$ and the upper bound on $e(G)$. 
We take Theorem \ref{mr} as an example to explain this trade-off relationship.
Let $c>0$ be a constant and let $k\ge2$ be an integer. Define
\begin{align*}
g(k,c) &=(2c+21)k^2+(c+\frac{25}{2})k,\\
 f(k,c) &= \frac{1}{ck}(14k^2-14k+1)\left((2c+21)k^2+(c+\frac{25}{2})k\right)\\
     &=\left(28 + \frac{294}{c}\right)k^{3} - \left(14 + \frac{119}{c}\right)k^{2} - \left(12 + \frac{154}{c}\right)k + 1 + \frac
{25}{2c}.
 \end{align*}
The same estimates show that Theorem \ref{mr} holds for a connected $G$ with $n\ge f(k,c)$ vertices and at most $n(1+1/g(k,c))$ edges.
For the sake of conciseness, we take $c =49$ and adopt suitably amplified lower bounds on $n$ to present Theorem \ref{mr}. 

Notably, the lower bound on $n$ and the upper bound on $e(G)$ are by no means tight. But significant improvement in these bounds will necessitate different methods.

 \section*{Acknowledgments}
This research was supported  by National Key R\&D Program of China under grant number 2024YFA1013900 and NSFC under grant number 12471327.

\end{document}